\documentclass[11pt]{amsart}

\usepackage{url,graphicx,tabularx,array,geometry}
\usepackage{graphicx}
\usepackage{amssymb}
\usepackage{fancyhdr}
\usepackage{wrapfig}
\usepackage{epstopdf}
\usepackage{amsmath}
\usepackage[labelfont=bf]{caption}
\usepackage[hidelinks]{hyperref}
\usepackage{amsthm}
\usepackage{bbm}
\usepackage[T1]{fontenc}
\usepackage{yfonts}
\usepackage{harpoon}
\usepackage[all]{xy}
\usepackage{tikz}
\usepackage{hyperref}
\usepackage{mathtools}
\hypersetup{
 colorlinks,
 linkcolor={red!50!black},
 citecolor={blue!50!black},
 urlcolor={blue!80!black}
}

\usepackage[maxbibnames = 50]{biblatex}
\renewbibmacro{in:}{}
\DeclareFieldFormat{pages}{#1}
\usetikzlibrary{decorations.markings}
\tikzstyle{vertex}=[circle, draw, inner sep=0pt, minimum size=6pt]

\makeatletter
\def\imod#1{\allowbreak\mkern10mu({\operator@font mod}\,\,#1)}
\makeatother

\theoremstyle{plain}

\newtheorem{theoremN}{Theorem}[section]

\newtheorem{corollaryN}[theoremN]{Corollary}
\newtheorem{lemmaN}[theoremN]{Lemma}

\theoremstyle{definition}

\theoremstyle{remark}

\newtheorem*{note}{Note}

\newcommand{\tightoverset}[2]{%
  \mathop{#2}\limits^{\vbox to -.5ex{\kern-0.75ex\hbox{$#1$}\vss}}}

\begin{document}
\title[The Adjacency Spectra of Some Families of Minimally Connected Prime Graphs]{The Adjacency Spectra of Some Families of Minimally Connected Prime Graphs}

\author[Florez, Higgins, Huang, Keller, and Shen]{Chris Florez, Jonathan Higgins, Kyle Huang, Thomas Michael Keller, Dawei Shen}

\address{Chris Florez, Department of Mathematics, David Rittenhouse Lab, University of Pennsylvania, 209 South 33rd Street, Philadelphia, PA 19104-6395, USA}
\email{cflorez@sas.upenn.edu}

\address{Jonathan Higgins, Mathematics and Computer Science Department, Wheaton College, 501 College Ave, Wheaton, IL 60187, USA}
\email{jonathan.higgins@my.wheaton.edu}

\address{Kyle Huang, Mathematics Department, University of California-Berkeley, 2227 Piedmont Avenue, Berkeley, CA 94709, USA}
\email{kyle.huang@berkeley.edu}

\address{Thomas Michael Keller, Department of Mathematics, Texas State University, 601 University Drive, San Marcos, TX 78666-4616, USA}
\email{keller@txstate.edu}

\address{Dawei Shen, Department of Mathematics and Statistics, Washington University in St. Louis, 1 Brookings Dr., St. Louis, MO 63105, USA}
\email{shen.dawei@wustl.edu}

\subjclass[2010] {Primary: 15A18, Secondary: 05C25}
\keywords {prime graph, adjacency matrix, spectral graph theory}
\maketitle

\begin{abstract}
    In finite group theory, studying the prime graph of a group has been an important topic
for almost the past half-century. Recently, prime graphs of solvable groups have been
characterized in graph theoretical terms only. This now allows the study of these 
graphs without any knowledge of the group theoretical background. In this paper we 
study prime graphs from a linear algebra angle and focus on the class of
minimally connected prime graphs introduced in earlier work on the subject. 
As our main results, we determine the 
determinants of the adjacency matrices and the spectra of some
important families of these graphs.

\end{abstract}

\begin{section}{Introduction}
This paper deals with prime graphs of finite solvable
groups. The prime graph of a finite group is the graph whose vertices are the prime numbers
dividing the order of the group, and two vertices are linked by an edge if and only if their 
product divides the order of some element of the group. Prime graphs were introduced by
Gruenberg and Kegel in the 1970s and have been an object of continuous study since then.
They were one of the first graphs assigned to groups. This idea of representing group theoretical
data via graphs and describing them via graph theoretical notions proved so successful 
that today there is a myriad of graphs (e.g. character degree graphs, conjugacy class size
graphs, etc.) and a whole industry of exploring them. For this reason, today prime graphs are
often referred to as Gruenberg-Kegel graphs.\\

While a focus in the study of prime graphs has been on simple groups for a long time,
the main result of \cite{2015_REU_Paper}, somewhat surprisingly, is a purely graph theoretical
characterization of prime graphs of solvable groups: A (simple) graph is the prime graph
of a finite solvable group if and only if its complement is triangle-free and 3-colorable.
This made it possible to study
simple groups whose prime graph is that of a finite solvable group, see 
\cite{Maslova_almost_simple}. Moreover, this characterization 
allowed the authors of \cite{2015_REU_Paper}, for solvable groups, to introduce and study 
the idea of
minimal prime graphs: connected graphs whose complement is triangle-free and 3-colorable,
but removing an edge means that the complement has a triangle or is no longer 3-colorable.
The groups whose prime graphs are minimal are groups which are "saturated in Frobenius actions"
and have a restricted, but highly non-trivial structure, as discussed in detail in 
\cite{2015_REU_Paper}. In \cite{mcpg_paper}
the authors thoroughly explore some graph theoretical properties of minimal prime graphs.
Also, an alternative notion of minimal prime graphs - minimally connected prime graphs -
is introduced and turns out to be closely related to minimal prime graphs.\\

In this paper, we study minimal prime graphs from a completely different angle, namely
 from a linear algebra point of view. This is one of the
first studies of graphs related to groups with a linear 
algebra focus (the only other group related graphs for which this has been done and that we are aware of being Cayley graphs). We will find a rich structure  
for some of the basic minimal and minimally connected prime graphs. We will study the 
determinants and sets 
of eigenvalues of the adjacency matrices (spectra) of such graphs and obtain some detailed
information. \\

We now explain the specific content of the paper in some more technical detail.\\

A minimal prime graph of a solvable group $\Gamma_G$ is defined as a connected graph of order $n>1$ such that $\Gamma_G \setminus \{pq\}$ is not the prime graph of a solvable group for any $pq \in E(\Gamma_G)$ (\cite{2015_REU_Paper}). Equivalently, a minimal prime graph is a connected graph of order $n>1$ whose complement is triangle-free and three-colorable, but the addition of an edge to its complement induces a triangle or renders it no longer three-colorable. A minimally connected prime graph of a solvable group is defined similarly, but it includes that the removal of any edge may result in a disconnected graph without inducing a triangle or changing the colorability of the complement (\cite{mcpg_paper}). If $\mathcal{G}$ is the class of minimal prime graphs and $\hat{\mathcal{G}}$ is the class of minimally connected prime graphs, it has been shown that $\mathcal{G} \subsetneq \hat{\mathcal{G}}$ and if $\Gamma \in \hat{\mathcal{G}}\setminus\mathcal{G}$, the $\Gamma$ is the graph of two vertex-disjoint complete graphs joined by one edge. These graphs are called complete bridge graphs, and we use the notation $B_{m,n}$ to denote the complete bridge graph of complete graphs $K_m$ and $K_n$ joined by one edge. For convenience, we say $m \geq n$, and in order to maintain the conditions in the definition of minimally connected prime graphs, we must enforce the conditions $m \geq n >1$, $m=2$ and $n=1$, or $m=n=1$ (\cite{mcpg_paper}). Another kind of minimally connected prime graph (which is contained in $\mathcal{G}$) that we will consider in this paper is reseminant graphs. Reseminant graphs are defined as graphs generated by repeated vertex duplication on $C_5$. Vertex duplication is a method of generating new graphs from old graphs, where if we have a graph $G$, we can produce the graph $G'$ by introducing a new vertex $v$ to $G$ and an edge $vv'$, where $v' \in V(G)$, along with the edges $vx$ if and only if $xv' \in E(G)$. In this paper, we are not concerned with the group theoretic problems, so when we say "minimally connected prime graphs," we are referring to minimally connected prime graphs of solvable groups. \bigskip

In Section 2, we will find formulae for the determinants of the adjacency matrices of complete bridge graphs, $B_{m,n}$, and suspension graphs $S(m,n)$. The suspension graph $S(m,n)$ is the minimal prime graph generated from $B_{m,n}$ by connecting a new vertex to every vertex in $B_{m,n}$ except for the two bridge vertices (\cite{mcpg_paper}). We use these results to determine the characteristic polynomials and the adjacency spectra of two particular families of minimally connected prime graphs in Sections 3 and 4. The first of these is the family of complete bridge graphs of the form $B_{m,m-1}$ where $m>2$, and the second is the family of graphs $\tilde{R}_n$, which denotes the unique reseminant graph of $n+5$ vertices acquired by repeatedly duplicating the same vertex in the induced 5-cycle. \bigskip

\begin{figure}
\label{dope graphs}
\centering
\scalebox{1.0}{
\begin{tikzpicture}

  [scale=.1,auto=left,every node/.style={circle,fill=black!20}]
  \node (n6) at (2,8.5) {};
  \node (n4) at (3,7.5)  {};
  \node (n5) at (4,7.5)  {};
  \node (n1) at (5,8.5) {};
  \node (n2) at (5,6.5)  {};
  \node (n3) at (2,6.5)  {};
  \node (n7) at (1,7.5) {};
  \node (n8) at (7,8.5) {};
  \node (n9) at (8,7.5)  {};
  \node (n10) at (9,7.5)  {};
  \node (n11) at (10,8.5) {};
  \node (n12) at (10,6.5)  {};
  \node (n13) at (7,6.5)  {};
  \node (n14) at (6,7.5) {};
  \node (n15) at (8.5,9) {};
  \node (n16) at (11.5,8.1) {};
  \node (n17) at (11.5,6.9) {};
  \node (n18) at (12.65,8.65) {};
  \node (n19) at (12.65,6.35) {};
  \node (n20) at (13.65,7.5) {};
  \node (n21) at (14.65,7.5) {};
  
\draw [fill=black] (2,8.5) circle (3.5pt);
\draw [fill=black] (3,7.5) circle (3.5pt);
\draw [fill=black] (4,7.5) circle (3.5pt);
\draw [fill=black] (5,8.5) circle (3.5pt);
\draw [fill=black] (5,6.5) circle (3.5pt);
\draw [fill=black] (2,6.5) circle (3.5pt);
\draw [fill=black] (1,7.5) circle (3.5pt);
\draw [fill=black] (7,8.5) circle (3.5pt);
\draw [fill=black] (8,7.5) circle (3.5pt);
\draw [fill=black] (9,7.5) circle (3.5pt);
\draw [fill=black] (10,8.5) circle (3.5pt);
\draw [fill=black] (10,6.5) circle (3.5pt);
\draw [fill=black] (7,6.5) circle (3.5pt);
\draw [fill=black] (6,7.5) circle (3.5pt);
\draw [fill=black] (8.5,9) circle (3.5pt);
\draw [fill=black] (11.5,8.1) circle (3.5pt);
\draw [fill=black] (11.5,6.9) circle (3.5pt);
\draw [fill=black] (12.65,8.65) circle (3.5pt);
\draw [fill=black] (12.65,6.35) circle (3.5 pt);
\draw [fill=black] (13.65,7.5) circle (3.5pt);
\draw [fill=black] (14.65,7.5) circle (3.5pt);

\node (n22) at (3.5,5.9) {\textbf{(i)}};
\node (n23) at (8.5,5.9) {\textbf{(ii)}};
\node (n23) at (12.75,5.9) {\textbf{(iii)}};

  \foreach \from/\to in {n6/n4,n4/n5,n5/n1,n1/n2,n2/n5,n6/n3,n3/n4,n7/n6,n7/n4,n7/n3,n14/n8,n14/n13,n14/n9,n9/n8,n9/n13,n8/n13,n9/n10,n10/n11,n10/n12,n11/n12,n15/n8,n15/n11,n15/n12,n15/n13,n15/n14,n16/n17,n16/n18,n17/n19,n18/n20,n19/n20,n18/n21,n19/n21,n20/n21}
    \draw (\from) -- (\to);
\end{tikzpicture}}

\caption{Here are examples of some of the types of graphs that we will consider in this paper. Graph (i) is the complete bridge graph $B_{4,3}$, graph (ii) is the suspension graph $S(4,3)$, and graph (iii) is the unique reseminant graph of 6 vertices.}
\end{figure}

Before proceeding, it will be helpful to define some additional notation that will be used. If we have a graph $\Gamma$, we will let $A(\Gamma)$ be the adjacency matrix of $\Gamma$. Also, we will let $I_n$ be the $n \times n$ identity matrix, $J_n$ be the $n \times n$ matrix of ones, and $(x)_n$ be the $n \times n$ matrix of x's, where $x \in \mathbb{R}$ (although $x$ will only be rational in this paper). We will also let $\hat{\mathcal{G}}$ denote the set of minimally connected prime graphs, $\mathcal{G}$ the minimal prime graphs, $\mathcal{R}$ the reseminant graphs, and $\tilde{\mathcal{R}}$ the reseminant graphs generated by repeatedly duplicating the same vertex.  
\end{section}

\begin{section}{Determinants of Adjacency Matrices for Some Minimally Connected Prime Graphs}

We will begin by considering $A(B_{m,n})$. It is easy to visualize the adjacency matrix of a complete bridge graph:

\begin{equation}
\label{adj matrix for complete bridge graphs}
A(B_{m,n}) = \left(\begin{array}{@{}c|c@{}}
  A(K_m)
  & \begin{matrix}
  0 & 0& \cdots &0\\
  \vdots & \vdots & \ddots & \vdots\\
  1 & 0 & \cdots & 0
  
  \end{matrix}\\
\hline
  \begin{matrix}
  0 & 0& \cdots &1\\
  \vdots & \vdots & \ddots & \vdots\\
  0 & 0 & \cdots & 0 \end{matrix} &
  A(K_n)
  
\end{array}\right),
\end{equation}

In the following theorem, we find the determinant of such a matrix.

\begin{theoremN}
The determinant of the adjacency matrix of a complete bridge graph is
\begin{equation}
\label{det of adj matrix for complete bridge graphs}
\det{ A(B_{m,n})} = (-1)^{m+n-1}(3-(m+n)).
\end{equation}
\end{theoremN}

\begin{proof}
It is a known result that if we have a block matrix
\[
M = \left(\begin{array}{@{}c|c@{}}
 A
  & B\\
\hline
  C &
  D
\end{array}\right),
\]

where $A$ is a $m \times m $ matrix, $B$ is a $m \times n$ matrix, $C$ is a $n \times m$ matrix, and $D$ is an invertible $n \times n$ matrix, then 
\[
\det{M} = \det{A-BD^{-1}C}\det{D}.
\]
The determinants for the adjacency matrices of complete graphs are well known. We know $\det{A(K_n)}=(-1)^{n-1}(n-1)$ (for example, see \cite{det_adj_matrices}). Given our equation above, this is equivalent to $\det{D}$ for what we are trying to prove. Now, we need to find $\det{A-BD^{-1}C}.$
\\
\\
Now $D = A(K_n)$, and so 
\[
D^{-1} = \left(\frac{1}{n-1}\right)_n - I_n.
\]
By Equation \ref{adj matrix for complete bridge graphs}, we see that we can let $B$ be the upper right sub-matrix, and we can let $C$ be the lower left sub-matrix. When we do so, we find
\[
BD^{-1}C = \left(\begin{matrix} -(\frac{n-2}{n-1}) & 0 & \cdots & 0 \\
0 & 0 & \cdots & 0\\
\vdots & \vdots & \ddots & \vdots\\
0 & 0 & \cdots & 0 \end{matrix}\right),
\]
from which we see
\[
A - BD^{-1}C = \left(\begin{array}{@{}c|c@{}}
  \frac{n-2}{n-1}
  & \begin{matrix}
  1 & 1& \cdots &1\\
  \end{matrix}\\
\hline
  \begin{matrix}
  1\\
  1\\ \vdots \\ 1
   \end{matrix} &
  A(K_{m-1})
  
\end{array}\right),
\]
which we know is a $m\times m$ matrix.
We will find the determinant of this matrix by also treating it as a block matrix. Let $A^*$ be the singleton matrix $\left(\frac{n-2}{n-1}\right)$, let $B^*$ and $C^*$ be the vectors of 1's on the top right and bottom left, and let $D^*= A(K_{m-1})$. Thus,  $\det{A - BD^{-1}C} = \det{A^*-B^*(D^*)^{-1}C^*}\det{
D^*}$. Similar to before, we find 
\[
(D^*)^{-1} = \left(\frac{1}{m-2}\right)_{m-1} - I_{m-1},
\]
and when we multiply this on the left and right by the 1-vectors $B^*$ and $C^*$, we are left with the singleton matrix $\left(\frac{m-1}{m-2}\right)$, from which we observe
\[
A^*-B^*(D^*)^{-1}C^* = \left(\frac{n-2}{n-1} - \frac{m-1}{m-2}\right),
\]
so 
\[
\det{A-BD^{-1}C} = (-1)^{m-2}(m-2)\, \left(\frac{n-2}{n-1} - \frac{m-1}{m-2}\right).
\]
From this, we can conclude 
\[
\det{A(\Gamma)} = (-1)^{m+n-3}(m-2)(n-1)\, \left(\frac{n-2}{n-1} - \frac{m-1}{m-2}\right) = (-1)^{(m+n)-1}(3-(m+n)),
\]
so the theorem follows.
\end{proof}

From this theorem, we also find the following corollary.

\begin{corollaryN}
$\det{A(B_{m,n})} = \det{A(B_{m',n'})}$ if and only if $m+n = m'+n'$.
\end{corollaryN}

\begin{proof}
The reverse direction is obvious. If $m+n = m'+n'$, then $(-1)^{(m+n)-1}(3-(m+n)) = (-1)^{(m'+n')-1}(3-(m'+n'))$, so it clearly follows that $\det{A(B_{m,n})} = \det{A(B_{m',n'})}$.
\\
\\
The equation in Theorem \ref{det of adj matrix for complete bridge graphs} can be thought of as a function of $m+n$. The forward direction follows from its being injective. 
\end{proof}

\begin{theoremN}
If $m+n>4$ or $m+n<4$, $\det{A(\overline{B_{m,n}})} = 0$. However, $\det{A(\overline{B_{2,2}})}=1$.
\end{theoremN}

\begin{proof}
We begin by noting the following equation, where $G_n$ is any simple graph on $n$ vertices:
\begin{equation}
\label{relating adjacency matrices b/w graph and complement}
A(\overline{G}_n) = A(K_n) - A(G_n).
\end{equation}
By Equations \ref{adj matrix for complete bridge graphs} and \ref{relating adjacency matrices b/w graph and complement}, we deduce
\begin{equation}
\label{adj matrix for complete bridge graph complements}
A(\overline{B_{m,n}}) = \left(\begin{array}{@{}c|c@{}}
  \text{\huge0}
  & \begin{matrix}
  1 & 1& \cdots &1\\
  \vdots & \vdots & \ddots & \vdots\\
  0 & 1 & \cdots & 1
  
  \end{matrix}\\
\hline
  \begin{matrix}
  1 & 1& \cdots &0\\
  \vdots & \vdots & \ddots & \vdots\\
  1 & 1 & \cdots & 1 \end{matrix} &
 \text{\huge0}
  
\end{array}\right).
\end{equation}
If $m+n>4$, then $m\geq 3$, so the first two columns of $A(\overline{B_{m,n}})$ are linearly independent, so the determinant is 0. If $m+n<4$, then $B_{m,n}$ is $B_{2,1}$ or $B_{1,1}$. It is easy to verify for both cases $\det{A(\overline{B_{m,n}})}=0$. If $m+n = 4$, then $B_{m,n} = B_{2,2}$, which is self-complementary (\cite{mcpg_paper}). Thus, $\det{A(\overline{B_{2,2}})} = \det{A(B_{2,2})} = -(3-4)=1$ by Theorem \ref{det of adj matrix for complete bridge graphs}.
\end{proof}

In general, the class of reseminant graphs (and certainly minimal prime graphs) is too broad to derive a formula for the determinant of their adjacency matrices. We can, however, find such formulae for specific kinds of reseminant graphs, such as suspension graphs. Our final goal for this sections is to derive such a formula, but this will require the help of the following lemma. 

\begin{lemmaN}
\label{THE lemma}
Consider a complete bridge graph $B_{m,n}$ and let $a_{ij}$ be the entry in the $i$th row and $j$th column of $A(B_{m,n})^{-1}$. Then, we find the following:

$a_{ij}=$
\begin{enumerate}
    \item $-(\dfrac{m+n-4}{m+n-3})$\,\, if $1 \leq i=j \leq m-1$ or $m+2 \leq i=j \leq m+n $
    \item $\dfrac{1}{m+n-3}$\,\, if $1 \leq i < j \leq m-1$, $1 \leq j < i \leq m-1$, $m+2 \leq i < j \leq m+n$, or $m+2 \leq j<i \leq m+n$
    \item $\dfrac{-1}{m+n-3}$ \,\, if $i \geq m+2$ and $j \leq m-1$, or $i \leq m-1$ and $j \geq m+2$.
\end{enumerate}
\end{lemmaN}

It may be noted that this does not cover every entry in $A(B_{m,n})^{-1}$, but it includes every entry that we will need to prove the theorem.

\begin{proof}
Throughout this proof, we will be using the equation
\[
a_{ij} = \frac{(adj\, A(B_{m,n}))_{ij}}{\det{A(B_{m,n})}}. 
\]
We already know $\det{A(B_{m,n})}$ from Theorem \ref{det of adj matrix for complete bridge graphs}.
We begin by showing that $a_{ij} = -(\frac{m+n-4}{m+n-3})$ if $1 \leq i=j \leq m-1$ or $m+2 \leq i=j \leq m+n $ in order to establish (i). Thus, we are looking at the diagonal of $A(B_{m,n})^{-1}$, excluding the entries where $i$ (and $j$) represent the vertex numbers for the bridge vertices. Also, 
\[
(adj\, A(B_{m,n}))_{ij} = (((-1)^{i+j} \det{M_{ij}})^T)_{ij},
\]
where $M_{ij}$ represents the minor matrix obtained after removing the $i$th row and $j$th column from $A(B_{m,n})$. Since we're on the diagonal, $i+j$ is even, so 
\[
(adj\, A(B_{m,n}))_{ij}\bigg\vert_{i=j} = \det{M_{ii}}.
\]
Given our conditions on $i$ and $j$, $M_{ii}$ is either $B_{m-1,n}$ or $B_{m,n-1}$, so $\det{M_{ii}} = (-1)^{m+n-2}(3-(m+n-1))$. Hence, we find, under these conditions,
\[
a_{ij} = \frac{(-1)^{m+n-2}(3-(m+n-1))}{(-1)^{m+n-1}(3-(m+n))} = - \frac{m+n-4}{m+n-3}.
\]
\\
\\
Next, we will prove (ii). Consider the condition $1 \leq i < j \leq m-1$, $1 \leq j < i \leq m-1$, $m+2 \leq i < j \leq m+n$, or $m+2 \leq j<i \leq m+n$. Either $i+j$ is even or odd.
\\
\textbf{Case 1.} Suppose $i+j$ is odd.
\\
For simplicity, we will begin by supposing $1 \leq i < j \leq m-1$. We know $a_{ij}=a_{ji}$ because the inverse of a symmetric matrix is also symmetric, which takes care of the condition $1 \leq j < i \leq m-1$. In this case, and when $i+j$ is odd, $M_{ij}$ is equal to $B_{m-1,n}$ with a 1 on the diagonal in the $(i-1)$th row. Let $B_{m-1,n}[e_{i-1}]_{i-1}$ be the matrix obtained by replacing the $(i-1)$th row in $B_{m-1,n}$ with the $(m +n-1) \times 1$ vector of 0s, except for a 1 in the $(i-1)$th position. Because $B_{m-1,n}$ and $B_{m-1,n}[e_{i-1}]_{i-1}$ differ only by one row and the sum of the two $(i-1)$th row vectors gives the $(i-1)$th row vector of $M_{ij}$, we find
 \[
 \det{M_{ij}} = \det{B_{m-1,n}} + \det{B_{m-1,n}[e_{i-1}]_{i-1}} = \det{B_{m-1,n}}+\det{B_{m-2,n}}
 \]\[
 = (-1)^{m+n-2}(3-(m+n-1)) + (-1)^{m+n-3}(3-(m+n-2)).
 \]
 If $m+n$ is even, $\det{M_{ij}} = -1$ and $\det{M_{ij}} = 1$ if $m+n$ is odd. Because of the symmetry of the inverse matrix, we will ignore the transpose in the equation for the adjoint. Thus, we find 
 \[
 (adj\, A(B_{m,n}))_{ij} = 1
 \]
 if $m+n$ is even and -1 if $m+n$ is odd. From this, we can deduce, either way, that
 \[
 a_{ij} = \frac{1}{m+n-3},
 \]
where we have used $\det{A(B_{m,n})}$ to finish evaluating $a_{ij}$. If we consider the $a_{ij}$ where $m+2 \leq i < j \leq m+n$ or $m+2 \leq j<i \leq m+n$, then we can use similar reasoning, but most of our adjustments will be in the $K_n$ sub-matrix rather than the $K_m$ submatrix.
\\
\textbf{Case 2.} Suppose $i+j$ is even. 
\\
In general, the work for this case will be similar to the work for Case 1. Thus, we will simplify the work again by supposing $1\leq i < j \leq m-1$. The main difference here is that $M_{ij}$ is not quite $B_{m-1,n}$ with a 1 on one of the diagonal entries. We can acquire this matrix, though, from swapping the $(i-1)$th row of $M_{ij}$ with the $(i-2)th$ row. Let this matrix be denoted $M_{ij}^*$. Of course, this requires that $m$ be sufficiently large. Thus, we know $\det{M_{ij}} = - \det{M_{ij}^*}$. Hence, we deduce from the results in the previous case that if $m+n$ is even, $\det{M_{ij}}=1$ and $\det{M_{ij}} = -1$ if $m+n$ is odd. Because $i+j$ is even, $(adj\, A(B_{m,n}) = \det{M_{ij}}$ (once again, we are ignoring the transpose because of the symmetry of the matrix), and, from here, we find once again
\[
a_{ij} = \frac{1}{m+n-3}.
\]
Just like in case 1, the same result can be obtained through similar reasoning for the other conditions on $i$ and $j$.
\\
Hence, we have proved (ii).
\\
\\
Finally, we want to prove (iii). Because of the symmetry of $A(B_{m,n})^{-1}$, we only need to consider the case where $i\geq m+2$ and $j \leq m-1$. 

We begin by considering the minor matrix $M_{ij}$, which is a $(m+n-1) \times (m+n-1)$ matrix. We can divide this into block matrices, where the upper left and bottom right sub-matrices are almost adjacency matrices of complete graphs. Let $A_{ij}$ be the $(m-1) \times (m-1)$ sub-matrix in the upper left, let $B_{ij}$ be the $(m-1) \times n$ sub-matrix of zeroes in the upper right, let $C_{ij}$ be $n \times (m-1)$ sub-matrix on the bottom left, and let $D_{ij}$ be the remaining $n \times n$ sub-matrix on the bottom right. Note that the top row of $D_{ij}$ will be a 1 followed by $n-1$ 0s. Also, $D_{ij}$ is invertible because its rows and columns are linearly independent, which is easy to verify. Thus, we know
\[
\det{M_{ij}} = \det{A_{ij}-B_{ij}D_{ij}^{-1}C_{ij}}\det{D_{ij}},
\]
which simplifies to 
\[
\det{M_{ij}} = \det{A_{ij}}\det{D_{ij}}
\]
because $B$ is a zero matrix. Additionally, if we let $D_{ij}^*$ be the $(n-1) \times (n-1)$ matrix that follows from removing the first row and column of $D_{ij}$, then we can simplify this further to
\[
\det{M_{ij}} = \det{A_{ij}}\det{D_{ij}^*}.
\]

Now, we will take a closer look at $A_{ij}$. We assert that it takes $m-(1+j)$ row swaps to get an adjacency matrix for a complete graph with a single 1 on the diagonal. For example, if $j=1$, meaning that we are looking at the $(m-1) \times (m-1)$ $A_{ij}$ matrix for some minor matrix $M_{i1}$ (the $i$ does not affect $A_{ij}$), then we find
\[
A_{i1} = \begin{matrix}
1 & 1 & 1 & \cdots & 1 & 1 \\
0 & 1 & 1 & \cdots & 1 & 1 \\
1 & 0 & 1 & \cdots & 1 & 1 \\
1 & 1 & 0 & \cdots & 1 & 1 \\
\vdots & \vdots & \vdots & \vdots & \vdots & \vdots \\
1 & 1 & 1 & \cdots & 0 & 1
\end{matrix}
\]
and we can get 
\[
A'_{i1} = \begin{matrix}
0 & 1 & 1 & \cdots & 1 & 1 \\
1 & 0 & 1 & \cdots & 1 & 1 \\
1 & 1 & 0 & \cdots & 1 & 1 \\
\vdots & \vdots & \vdots & \vdots & \vdots & \vdots \\
1 & 1 & 1 & \cdots & 0 & 1\\
1 & 1 & 1 & \cdots & 1 & 1 
\end{matrix}
\]
by $m-2$ row swaps. Hence, we can deduce
\[
\det{A_{ij}} = (-1)^{m-(1+j)}\left(\det{A(K_{m-1})} + \det{A(K_{m-1})[e_{m-1}]_{m-1}}\right)
\]\[
 = (-1)^{m-(1+j)}\left(\det{A(K_{m-1})} + \det{A(K_{m-2})}\right) = (-1)^{m-(1+j)} (-1)^m = (-1)^{1+j},
\]
where $A(K_{m-1})[e_{m-1}]_{m-1}$ represents the matrix obtained by replacing the $(m-1)$th row of $A(K_{m-1})$ with the vector containing all 0s except for a 1 in the $(m-1)$th position.

Now, we want to shift our focus to $D_{ij}$. As we said above, $\det{D_{ij}} = \det{D_{ij}^*}$ because the first row of $D_{ij}$ consists of a 1 followed by 0s, so we can actually focus on $D_{ij}^*$. This time, we assert (and leave to the reader to verify) that it takes $i-(m+2)$ swaps to make $D_{ij}^*$ equivalent to the adjacency matrix of a complete graph (on $n-1$ vertices) with a 1 on the diagonal. This time, $D_{ij}$ is unaffected by $j$. We find
\[
\det{D_{ij}} = \det{D_{ij}^*} = (-1)^{i-(m+2)}\left(\det{A(K_{n-1})} + \det{A(K_{n-1})[e_{i-1}]_{i-1}}\right)
\]\[
= (-1)^{i-(m+2)}\left(\det{A(K_{n-1})} +\det{A(K_{n-2})}\right) = (-1)^{i-(m+2)} (-1)^n,
\]
which we can simply write as $\det{D_{ij}} = (-1)^{i-(m+n)}$ because our base is -1.

Thus, we can conclude
\[
\det{M_{ij}} = (-1)^{1+j} (-1)^{i-(m+n)} = (-1)^{1+(i+j)-(m+n)}.
\]
As before, we will ignore the transpose in the adjoint because our matrix is symmetric. We find
\[
(adj\, A(B_{m,n}))_{ij} = (-1)^{i+j} (-1)^{1+(i+j)-(m+n)} = (-1)^{1-(m+n)} = (-1)^{1+(m+n)}.
\]
Now, we can calculate $a_{ij}$:
\[
a_{ij} = \frac{(-1)^{1+m+n}}{(-1)^{m+n-1}(3-(m+n))} = \frac{1}{3-(m+n)} = \frac{-1}{m+n-3}.
\]
This is what we sought to prove.
\end{proof}

Now, with Lemma \ref{THE lemma}, we are equipped to prove the following theorem.

\begin{theoremN}
\label{determinant C -> G}
The formula for the determinant of the adjacency matrix of a suspension graph is
\begin{equation}
    \det{A(S(m,n))} = (-1)^{m+n}(4mn-5(m+n)+6)
\end{equation}
\end{theoremN}

\begin{proof} 
If $\Gamma \in \mathcal{G}$ is generated from a complete bridge graph $B_{m,n}$, we know
\[
A(S(m,n)) = \left(\begin{array}{@{}c|c@{}}
 A(B_{m,n})
  & \vec{v}\\
\hline
  \vec{v}^T &
  0
\end{array}\right),
\]
where $\vec{v}$ is the vector of 1s, except for two 0s in the coordinate positions of the bridge vertices. Also,
\[
\left(\begin{array}{@{}c|c@{}}
 A(B_{m,n})
  & \vec{v}\\
\hline
  \vec{v}^T &
  0
\end{array}\right) \left(\begin{array}{@{}c|c@{}}
 I
  & -A(B_{m,n})^{-1} \vec{v}\\
\hline
  0 &
  I
\end{array}\right) = \left(\begin{array}{@{}c|c@{}}
 A(B_{m,n})
  & \vec{0}\\
\hline
  \vec{v}^T &
  -\vec{v}^T A(B_{m,n})^{-1} \vec{v}
\end{array}\right),
\]
from which it follows
\[
\det{A(\Gamma)} = \det{A(B_{m,n})} \det{-\vec{v}^T A(B_{m,n})^{-1} \vec{v}} = - \det{A(B_{m,n})} \det{\vec{v}^T A(B_{m,n})^{-1} \vec{v}}.
\]
We already know $\det{A(B_{m,n})}$. Thus, if $a_{ij}$ are the entries of $A(B_{m,n})^{-1}$, we can deduce
\[
\det{\vec{v}^T A(B_{m,n})^{-1} \vec{v}} = \sum_{\substack{1\leq i,j \leq m-1 \\\textrm{or}\, m+2 \leq i,j \leq m+n}}a_{ij},
\]
and this summation covers all of the entries considered in Lemma \ref{THE lemma}. When we add all these values together, we get
\[
\det{\vec{v}^T A(B_{m,n})^{-1} \vec{v}} = \frac{1}{m+n-3}(-(m+n-4)(m+n-2)\]\[ + (m-1)(m-2)+(n-1)(n-2)-2(m-1)(n-1)).
\]
Thus, we find
\[
\det{A(S(m,n))} = \frac{(-1)^{m+n-1}(3-(m+n))}{m+n-3}((m+n-4)(m+n-2)\]\[ + (m-1)(2-m)+(n-1)(2-n)+2(m-1)(n-1)),
\]
which simplifies to 
\[
\det{A(S(m,n))} = (-1)^{m+n}(4mn-5(m+n)+6),
\]
and this is what we were trying to prove.
\end{proof}

Before moving on, it may help to consider an example to show how the last two results work. Consider $S(4,3)$, the suspension graph generated by $B_{4,3}$ via the Suspension Method (these are graphs (ii) and (i) in Figure 1, respectively). The adjacency matrix of $S(4,3)$ is
\[
A(S(4,3))=
\left(\begin{matrix}
0 & 1 & 1 & 1 & 0 & 0 & 0 & 1\\
1 & 0 & 1 & 1 & 0 & 0 & 0 & 1\\
1 & 1 & 0 & 1 & 0 & 0 & 0 & 1\\
1 & 1 & 1 & 0 & 0 & 0 & 0 & 0\\
0 & 0 & 0 & 0 & 0 & 1 & 1 & 0\\
0 & 0 & 0 & 0 & 1 & 0 & 1 & 1\\
0 & 0 & 0 & 0 & 1 & 1 & 0 & 1\\
1 & 1 & 1 & 0 & 0 & 1 & 1 & 0
\end{matrix}\right),
\]
and the adjacency matrix of $B_{4,3}$ and its inverse are
\[
A(B_{4,3})=
\left(\begin{matrix}
0 & 1 & 1 & 1 & 0 & 0 & 0 \\
1 & 0 & 1 & 1 & 0 & 0 & 0 \\
1 & 1 & 0 & 1 & 0 & 0 & 0 \\
1 & 1 & 1 & 0 & 1 & 0 & 0 \\
0 & 0 & 0 & 1 & 0 & 1 & 1 \\
0 & 0 & 0 & 0 & 1 & 0 & 1 \\
0 & 0 & 0 & 0 & 1 & 1 & 0 
\end{matrix}\right), \,
A(B_{4,3})^{-1}=
\left(\begin{matrix}
-\frac{3}{4} & \frac{1}{4} & \frac{1}{4} & \frac{1}{2} & \frac{1}{4} & -\frac{1}{4} & -\frac{1}{4} \\
\frac{1}{4} & -\frac{3}{4} & \frac{1}{4} & \frac{1}{2} & \frac{1}{4} & -\frac{1}{4} & -\frac{1}{4} \\
\frac{1}{4} & \frac{1}{4} & -\frac{3}{4} & \frac{1}{2} & \frac{1}{4} & -\frac{1}{4} & -\frac{1}{4} \\
\frac{1}{2} & \frac{1}{2} & \frac{1}{2} & -1 & -\frac{1}{2} & \frac{1}{2} & \frac{1}{2} \\
\frac{1}{4} & \frac{1}{4} & \frac{1}{4} & -\frac{1}{2} & -\frac{3}{4} & \frac{3}{4} & \frac{3}{4} \\
-\frac{1}{4} & -\frac{1}{4} & -\frac{1}{4} & \frac{1}{2} & \frac{3}{4} & -\frac{3}{4} & \frac{1}{4} \\
-\frac{1}{4} & -\frac{1}{4} & -\frac{1}{4} & \frac{1}{2} & \frac{3}{4} & \frac{1}{4} & -\frac{3}{4} 
\end{matrix}\right)
\]
The reader is encouraged to verify Lemma \ref{THE lemma} with $A(B_{4,3})^{-1}$. Additionally, one can easily use a CAS to determine $\det{A(S(4,3))}=-19$, which agrees with Theorem \ref{determinant C -> G}. 
\end{section}

\begin{section}{The Spectrum of $B_{m,m-1}$}

In this section, we will determine the characteristic polynomial of complete bridge graphs $B_{m,m-1}$ and the ordering of the eigenvalues. In general, we will let $\lambda_i$ denote the $i$th eigenvalue, where $\lambda_1 \geq \lambda_2 \geq ... \geq \lambda_{2m-1}$. First, we will establish some lemmas that are well-known in matrix and spectral graph theory.

\begin{lemmaN}
\label{PF-lemma}
\textbf{(Perron-Frobenius)} A real square matrix with all positive entries has a unique largest eigenvalue $\lambda_1$, and for any other eigenvalue $\lambda$, $|\lambda|<\lambda_1$. (See \cite{Perron_Frobenius})
\end{lemmaN}

\begin{lemmaN}
\label{eigenvalue multiplicity}
For a graph $\Gamma$ on $n$ vertices, the multiplicity of an eigenvalue $\lambda$ is $n-$rank$(I\lambda - A(\Gamma))$.
\end{lemmaN}

Now, we can use these results to start proving some new things about the spectrum of $B_{m,m-1}.$ We will consider the eigenvalue $-1$ extensively in this paper, so instead of saying rank$(-I-A(\Gamma))$, we will say rank$(I+A(\Gamma))$, for the two are identical.

\begin{lemmaN}
\label{multiplicity of -1 eigenvalues for A(B_m,n)}
If $m+n > 4$, $A(B_{m,n})$ has an eigenvalue -1 with multiplicity $m+n-4$.
\end{lemmaN}

Recall that we let $m \geq n$ by convenction, so if $m+n>4$, then $m\geq 3$.

\begin{proof}
By Lemma \ref{eigenvalue multiplicity}, we find that -1 has a multiplicity of 
\[
m+n-\textrm{rank}(I+A(B_{m,n})).
\]
It is not difficult to see
\[
I+A(B_{m,n}) = \left(\begin{array}{@{}c|c@{}}
  J_m
  & \begin{matrix}
  0 & 0& \cdots &0\\
  \vdots & \vdots & \ddots & \vdots\\
  1 & 0 & \cdots & 0
  
  \end{matrix}\\
\hline
  \begin{matrix}
  0 & 0& \cdots &1\\
  \vdots & \vdots & \ddots & \vdots\\
  0 & 0 & \cdots & 0 \end{matrix} &
  J_n
  
\end{array}\right).
\]
Clearly, this matrix only has four linearly independent rows (and columns)-- we can choose the $m$th row, the $(m+1)$th row, and any single row between 1 and $m-1$ and between $m+2$ and $m+n$. Hence, rank($I+A(B_{m,n}))=4$, so the lemma follows. 
\end{proof}

Thus, we know all but four of the eigenvalues of any complete bridge graph $B_{m,n}$. Additionally, if we let $\lambda^*_i$ be these remaining eigenvalues that are not -1, where $1 \leq i \leq 4$ and $\lambda^*_i \geq \lambda^*_{i+1}$, we can use the fact that the product of all the eigenvalues of an adjacency matrix is its determinant and the sum of the eigenvalues is the trace (which is 0) to deduce $\prod_{i=1}^4 \lambda^*_i = m+n-3$ and $\sum_{i=1}^4 \lambda^*_i = m+n-4$. Hence, we find the equation
\[
\prod_{i=1}^4 \lambda^*_i-\sum_{i=1}^4 \lambda^*_i = 1. 
\]
We can, however, say much more about these remaining eigenvalues. In the following lemma, we find upper and lower bounds for the greatest eigenvalue $\lambda_1$ for any complete bridge graph before focusing specifically on $B_{m,m-1}$.

\begin{lemmaN}
\label{bounds of lambda_1}
Let $\lambda_1$ be the greatest eigenvalue of $A(B_{m,n})$. Then,
\begin{equation}
    m-1 \leq \lambda_1 \leq m.
\end{equation}
If $m = n$, then
\begin{equation}
    m-\left(1-\frac{1}{m}\right) \leq \lambda_1 \leq m.
\end{equation}
\end{lemmaN}

\begin{proof}
This follows from the fact that the greatest eigenvalue must be bounded below by the average vertex degree of every induced subgraph and bounded above by the maximum vertex degree. We know the maximum degree is $m$, so $\lambda_1 \leq m$. We can find the tightest lower bound by considering the average degree on the complete induced subgraph of $m$ vertices, which is just $m-1$, so this establishes the lower bound. 
\\
\\
Now, we just need to establish the lower bound in the special case where $m=n$. This lower bound is tighter because the average degree of the whole graph is greater than the average degree of either of the two complete induced subgraphs because there are now two vertices with $m$ edges, whereas they all have $m-1$ edges in the subgraphs. We calculate the average degree as follows:
\[
d_{avg} = \frac{2(m-1)^2+2m}{2m} = \frac{2m^2-2m+2}{2m} = m - \left(1-\frac{1}{m}\right),
\]
which is what we wanted to show.
\end{proof}

From this, we know $\lambda^*_1 =\lambda_1$ because clearly the greatest eigenvalue must be one of the four eigenvalues that are not -1. Next, in the following lemma, we find that the second largest eigenvalue, $\lambda_2$, can be determined exactly for $B_{m,m-1}$.

\begin{lemmaN}\label{lambda_2 of case n=m-1}
If $m>2$, then $A(B_{m,m-1})$ has the eigenvalue $\lambda_2 = m-2$.
\end{lemmaN}

\begin{proof}
In this proof, we will use $[(m-2)I-A(B_{m,m-1})]_i$ to represent the $i$th row of the matrix $(m-2)I-A(B_{m,m-1})$. First, we note
\begin{eqnarray*}
\frac{1}{m-1}\sum_{i=1}^{m-1} [(m-2)I-A(B_{m,m-1})]_i - \sum_{j=m+1}^{m+n-1}[(m-2)I-A(B_{m,m-1})]_j 
\\
= [(m-2)I-A(B_{m,m-1})]_{m+n},
\end{eqnarray*}
which shows that the rows of $(m-2)I-A(B_{m,m-1})$ are not linearly independent. Thus, we know $\det{(m-2)I-A(B_{m,m-1})}=0$, so we can deduce that $(m-2)$ is an eigenvalue $\lambda$ of $A(B_{m,m-1})$. We also observe $\lambda = m-2 < m-1 \leq \lambda_1$, so we know that $\lambda$ cannot be the greatest eigenvalue. Because $m-2>-1$, we know at least that $\lambda^*_2 = \lambda_2$. Hence, $\lambda$ is either $\lambda_2$, $\lambda^*_3$, or $\lambda^*_4$.
\\
\\
Assume $\lambda$ is not $\lambda_2$. Then there is another $\lambda' = \lambda_2$ such that $m-2 < \lambda' < \lambda_1$. We know $\sum_{i=1}^4 \lambda^*_i = 2m-5$, and
\[
m-1+2(m-2)+\lambda^*_4 = 3m-5 + \lambda^*_4 < \sum_{i=1}^4 \lambda^*_i = 2m-5,
\]
so we know $\lambda_4<-m$. This cannot be, though, because the absolute value of the most negative eigenvalue must be less than $|\lambda_1|$ by Lemma \ref{PF-lemma}. Thus, we have a contradiction, and $\lambda = \lambda_2$.
\end{proof}

\begin{corollaryN}
\label{bounds of lambda_3 and lambda_4 B_m,m-1}
If $m>2$, then $\lambda_3$ and $\lambda_4$ of $A(B_{m,m-1})$ satisfy the following inequalities: \begin{enumerate}
    \item $-3 \leq \lambda^*_3 + \lambda^*_4 \leq -2$
    \item $\frac{2}{m} \leq \lambda^*_3 \lambda^*_4 \leq \frac{2}{m-1}$.
\end{enumerate} 
\end{corollaryN}

\begin{proof}
This follows from Lemmas \ref{bounds of lambda_1}, \ref{lambda_2 of case n=m-1}, and the equations for the sums and products of the eigenvalues.
\end{proof}

Now, we are prepared to determine the characteristic polynomial and spectrum of $B_{m,m-1}$. For the sake of simplicity, we will let $\mu_1$, $\mu_2$, and $\mu_3$ represent the remaining unknown eigenvalues (corresponding to $\lambda_1, \lambda^*_3,$ and $\lambda^*_4$). 

\begin{theoremN}
If $m>2$,the characteristic polynomial of $B_{m,m-1}$ is
\begin{equation}
\label{characteristic polynomial equation for B_m,m-1}
    \phi(A(B_{m,m-1}),x) = (x^3+(3-m)x^2 + (2-2m)x-2)(x-(m-2))(x+1)^{2m-5}.
\end{equation}
\end{theoremN}

\begin{proof}
We know $(x-(m-2))$ and $(x+1)^{2m-5}$ divide $\phi(A(B_{m,m-1}),x)$ by Lemmas \ref{multiplicity of -1 eigenvalues for A(B_m,n)} and \ref{lambda_2 of case n=m-1}. Next, we will calculate $\sum_{i=1}^3 \mu_i$, $\sum_{1 \leq i < j \leq 3} \mu_i \mu_j$, and $\prod_{i=1}^3 \mu_i$, because 
\[\prod_{i=1}^3 (x-\mu_i) = x^3 - \left(\sum_{i=1}^3 \mu_i\right)x^2 + \left(\sum_{1 \leq i < j \leq 3} \mu_i \mu_j\right)x -  \prod_{i=1}^3 \mu_i. \]
If we determine this, then we have the entire characteristic polynomial
\\
\\
First, we know $\sum_{i=1}^n \lambda_i=0$. Also,
\[
\sum_{i=1}^n \lambda_i = \sum_{i=1}^3 \mu_i - (2m-5)+m-2=-m+3=0,
\]
so it follows that $\sum_{i=1}^3 \mu_i = m-3$. Next, we observe 
\[
\sum_{1\leq i < j \leq n}\lambda_i \lambda_j = \frac{1}{2}\left(\left(\sum_{i=1}^n \lambda_i\right)^2 - \sum_{i=1}^n \lambda_i^2\right) = \frac{1}{2}(0-2e(B_{m,m-1}) = -e(B_{m,m-1}),
\]
where $e(B_{m,m-1})$ represents the number of edges in the graph $B_{m,m-1}$, which is not a difficult quantity to determine. We find $e(B_{m,m-1}) = m^2-2m+2$, so
\[
\sum_{1 \leq i < j \leq n} \lambda_i \lambda_j = \sum_{1 \leq i < j \leq 3} \mu_i \mu_j+((m-2)-(2m-5))\sum_{i=1}^3 \mu_i -(m-2)(2m-5)\]\[+\frac{1}{2}(2m-6)(2m-5) = \sum_{1 \leq i < j \leq 3} \mu_i \mu_j-(m-3)^2 -(m-2)(2m-5)\]\[+\frac{1}{2}(2m-6)(2m-5) =  -m^2+2m-2,
\]
and from here, we can show $\sum_{1 \leq i < j \leq 3} \mu_i \mu_j = 2-2m$. We also know $\prod_{i=1}^n \lambda_i = (-1)^{2m-2}(3-(2m-1)) = (4-2m)$ from Theorem \ref{det of adj matrix for complete bridge graphs}, and
\[
\prod_{i=1}^n \lambda_i = \prod_{i=1}^3 \mu_i(-1)^{2m-5}(m-2) = 4-2m,
\]
so we determine $\prod_{i=1}^3 \mu_i =2$ because we know $m>2$ (otherwise, this quantity could be undefined). Thus, we see $\prod_{i=1}^3 (x-\mu_i) = (x^3+(3-m)x^2 + (2-2m)x-2)$, so $\prod_{i=1}^n (x-\lambda_i)$ is Equation \ref{characteristic polynomial equation for B_m,m-1}, and this is the characteristic polynomial of $A(B_{m,m-1})$. This completes the proof.
\end{proof}

We can also determine the spectrum of $B_{m,m-1}$. Let $\theta_1 \geq \theta_2 \geq \theta_3$ be the roots of $x^3+(3-m)x^2 + (2-2m)x-2$. 

\begin{theoremN}
If $m>2$, then $\theta_1 > m-2>0>\theta_2>-1>\theta_3$, and
\[
Spec(B_{m,m-1}) = \left(\begin{matrix}
\theta_1 & m-2 & \theta_2 & -1 & \theta_3\\
1 & 1 & 1 & 2m-5 & 1
\end{matrix}\right).
\]
\end{theoremN}

\begin{proof}
By Lemma \ref{bounds of lambda_1}, we know that the greatest eigenvalue must be at least $m-1$, so we know that this greatest eigenvalue is $\theta_1$ and $\theta_1> m-2$. By Corollary \ref{bounds of lambda_3 and lambda_4 B_m,m-1} (i), we see that at least $\theta_3$ is negative and by (ii), we know that $\theta_2$ and $\theta_3$ must both be negative. Assume $\theta_2 > \theta_3>-1$. Then $\theta_2+\theta_3 > -2$, which contradicts Corollary \ref{bounds of lambda_3 and lambda_4 B_m,m-1} (i). If $-1>\theta_2>\theta_3$, then $\theta_2 \theta_3 >1$, which contradicts Corollary \ref{bounds of lambda_3 and lambda_4 B_m,m-1} (ii). If we assume $\theta_2 = \theta_3 =-1$, then the only $m>2$ that satisfies Corollary \ref{bounds of lambda_3 and lambda_4 B_m,m-1} is $m=3$, and it is easy to determine that $-1$ is not a root of $x^3-4x-2$. If $\theta_2>\theta_3=-1$, then (i) is not satisfied, and if $\theta_2= -1>\theta_3$, then (ii) is not satisfied (of Corollary \ref{bounds of lambda_3 and lambda_4 B_m,m-1}). Thus, by this casework, we know $0>\theta_2>-1>\theta_3$, which completes the proof.
\end{proof}

\end{section}

\begin{section}{The Spectrum of a Special Class of Reseminant Graphs}

As in Section 3, we will begin this section by determining the multiplicity of -1 as an eigenvalue for any reseminant graph. Recall from the Introduction that $\mathcal{R}$ denotes the set of reseminant graphs.

\begin{lemmaN}
\label{-1 eigenvalue reseminant graphs}
If $\Gamma \in \mathcal{R}$ and $|V(\Gamma)|>5$, then $\Gamma$ has -1 as an eigenvalue of multiplicity $|V(\Gamma)|-5$. If $\Gamma = C_5$, then $\Gamma$ does not have -1 as an eigenvalue.
\end{lemmaN}

\begin{proof}
Consider a reseminant graph $\Gamma$. A graph theoretic interpretation of the matrix $I+A(\Gamma)$ is that it is the graph $\Gamma$ with loops drawn at all of its vertices. In other words, each vertex is included in its own neighborhood. Because $\Gamma$ is reseminant, it is generated by duplicating (i.e. applying vertex duplication to) the vertices of the 5-cycle. Thus, in $I+A(\Gamma)$, we see that for any vertex $v$ in the induced 5-cycle of $\Gamma$, $v$ shares the same neighborhood as any vertex resulting from duplicating $v$. This partitions $\Gamma$ into five collections of vertices, and we find that the rows and columns corresponding to the vertices in each of these five sets are linearly dependent. Hence, rank$(I+A(\Gamma))=5$. Because rank$(I+A(\Gamma))=\textrm{rank}(-I-A(\Gamma))$, we deduce from Lemma \ref{eigenvalue multiplicity} that the multplicity of the -1 eigenvalue is $|V(\Gamma)|-5$.
\end{proof}
\end{section}

Before we can shift our focus to the special family of reseminant graphs, we must provide a new definition. First, $K^-_n$ is often used to denote a complete graph with an edge removed, and we shall adopt that notation.

\bigskip
\noindent
\textbf{Definition.}
An induced subgraph $\Gamma[\pi]$ of $\Gamma$ is a \textbf{maximal $K^-_n$ induced subgraph} of $\Gamma$ if $|\pi| = n$ and $\Gamma[\pi]$ is a $K^-_n$ induced subgraph of $\Gamma$ on the set $\pi$, where $\pi \subset V(\Gamma)$, but $\Gamma[\pi \cup \{v\}]$ is not a $K^-_{n+1}$ induced subgraph for any $v \in V(\Gamma)\,\setminus\, \pi$.

\bigskip
Now, we are prepared to start studying $\tilde{R}$, the family of reseminant graphs generated by repeatedly duplicating the same vertex of the induced 5-cycle. 

\begin{lemmaN}
\label{reseminant almost complete induced subgraph lemma}
If $\Gamma \in \mathcal{R}$, then $\Gamma \in \tilde{\mathcal{R}}$ if and only if $\Gamma$ has no more than one maximal $K^-_{i\geq4}$ induced subgraphs.
\end{lemmaN}

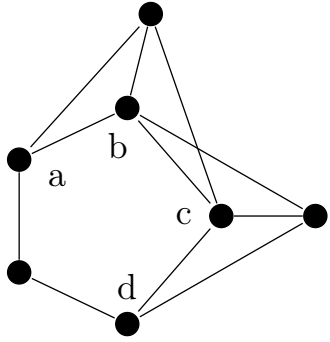
\begin{figure}
   \label{lemma 4.2 figure}
\centering
\scalebox{1.25}{
\begin{tikzpicture}
 
    \node (n16) at (11.5,8.1) {};
  \node (n17) at (11.5,6.9) {};
  \node (n18) at (12.65,8.65) {};
  \node (n19) at (12.65,6.35) {};
  \node (n20) at (13.65,7.5) {};
  \node (n21) at (14.65,7.5) {};
  \node (n22) at (12.9,9.65) {};
  
  \draw [fill=black] (11.5,8.1) circle (3.5pt);
\draw [fill=black] (11.5,6.9) circle (3.5pt);
\draw [fill=black] (12.65,8.65) circle (3.5pt);
\draw [fill=black] (12.65,6.35) circle (3.5 pt);
\draw [fill=black] (13.65,7.5) circle (3.5pt);
\draw [fill=black] (14.65,7.5) circle (3.5pt);
\draw [fill=black] (12.9,9.65) circle (3.5pt);

\node (n23) at (12.55,8.25) {b};
\node (n24) at (11.9,7.9) {a};
\node (n25) at (13.25,7.5) {c};
\node (n26) at (12.65,6.75) {d};

\foreach \from/\to in {n16/n17,n16/n18,n17/n19,n18/n20,n19/n20,n18/n21,n19/n21,n20/n21,n22/n16,n22/n18,n22/n20}
    \draw (\from) -- (\to);
\end{tikzpicture}}
\caption{The graph above is a reseminant graph on 7 vertices, where two distinct vertices are duplicated. Using the notation in the proof of Lemma \ref{reseminant almost complete induced subgraph lemma}, we can let $a=v^*_1$, $b = v^*=v'_1$, $c = v'=v^*_2$, and $d = v'_2$.}
\end{figure}

\begin{proof}
Let $\Gamma \in \mathcal{R}$. The forward direction follows from the definition of $\tilde{\mathcal{R}}$. Next, assume that $\Gamma$ has no more than one maximal $K^-_{i \geq 4}$ induced subgraph, but $\Gamma$ is not $C_5$ or some other reseminant graph generated by repeatedly duplicating the same vertex in the induced subgraph isomorphic to $C_5$. Thus, we know that two unique vertices $v^*$ and $v'$ on the 5-cycle were duplicated in the construction of $\Gamma$. Let $v^*_1$ and $v^*_2$ be the vertices in the induced 5-cycle of $\Gamma$ that are adjacent to $v^*$, and similarly define $v'_1$ and $v'_2$. Note that not all of these six vertices will be distinct because they are all on the induced 5-cycle and that $v^*$ and $v'$ might be adjacent. Refer to Figure 2 to help visualize how these vertices are labelled. Additionally, let $V^*$ be the set of vertices in $\Gamma$ that are the result of duplicating $v^*$, and define $V'$ similarly. We can say that some vertex $v$ is in $V^*$ if and only if the induced subgraph on $v, v^*, v^*_1,$ and $ v^*_2$ is $K^-_4$.
\\
\\
First, note that $\Gamma[V^* \cup \{v^*,v^*_1, v^*_2\}]$ and $\Gamma[V' \cup \{v',v'_1, v'_2\}]$ are $K^-_{|V^*|+3}$ and $K^-_{|V'|+3}$ induced subgraphs. By the method of vertex duplication, these are both almost complete subgraphs, except that they are missing the edges $v^*_1v^*_2$ and $v'_1v'_2$, respectively. We can also show that they are both maximal. Consider, for example, $\Gamma[V^* \cup \{v^*,v^*_1, v^*_2\}]$. We need to show that $\Gamma[V^* \cup \{v^*,v^*_1, v^*_2\} \cup \{\tilde{v}\}]$ is not some $K^-_{|V^*|+4}$ induced subgraph for any $\tilde{v} \in V(\Gamma_n) \setminus (V^* \cup \{v^*,v^*_1, v^*_2\})$. First of all, note that there is no vertex whose duplication results in the edge $v^*_1v^*_2$ because $d(v^*_1,v^*_2)=2$. Additionally, any vertex $\tilde{v} \in V(\Gamma_n) \setminus (V^* \cup \{v^*,v^*_1, v^*_2\})$ is a distance of 2 away from $v^*_1$ or $v^*_2$, which we can show as follows. If $d(v^*_1,\tilde{v})=1$, then $d(v^*_2,\tilde{v})\leq 2$ because the diameter of a reseminant graph is always 2 (as is easy to see). Assume $d(v^*_2,\tilde{v})=1$. Then, the vertices $v^*, v^*_1, v^*_2,$ and $\tilde{v}$ form a $K^-_4$ induced subgraph. However, this implies $\tilde{v} \in V^*$, which is a contradiction because we said $\tilde{v} \in V(\Gamma_n) \setminus (V^* \cup \{v^*,v^*_1, v^*_2\})$. Thus, $d(v^*_2,\tilde{v})=2$. Similarly, we can deduce that $d(v^*_1,\tilde{v})=2$ if $d(v^*_2,\tilde{v})=1$. We have found that $\Gamma[V^* \cup \{v^*,v^*_1, v^*_2\} \cup \{\tilde{v}\}]$ has at least two edges less than $K_{|V^*|+4}$, so it cannot be some $K^-_{|V^*|+4}$ induced subgraph. The same result can be found by considering $\Gamma[V' \cup \{v',v'_1, v'_2\}]$. Hence, we know that $\Gamma[V^* \cup \{v^*,v^*_1, v^*_2\}]$ and $\Gamma[V' \cup \{v',v'_1, v'_2\}]$ are both maximal $K^-_{i\geq4}$ induced subgraphs.
\\
\\
We have seen that if a reseminant graph can be generated by duplicating two distinct vertices in the induced subgraph isomorphic to $C_5$, then we have at least two $K^-_{i\geq 4}$ induced subgraphs. This also holds for repeatedly duplicating distinct vertices and for duplicating more than two vertices in the 5-cycle. This contradicts our original assumption. $C_5$ has no (maximal) $K^-_{i\geq 4}$ induced subgraphs, and it is not difficult to see that the reseminant graphs generated by repeatedly duplicating the same vertex in the induced 5-cycle have exactly one maximal $K^-_{i\geq 4}$ induced subgraph, so this establishes the lemma.
\end{proof}




Now, we have determined an important property of $\tilde{R}$ that does not apply to any other reseminant graphs. Also, we can consider the following corollary from Lemma \ref{reseminant almost complete induced subgraph lemma}, which is helpful for visualizing the reseminant graphs to which it applies. It will be stated without proof, as its truth should be obvious.

\begin{corollaryN}
\label{degree 2 corollary}
If $\Gamma \in \mathcal{R}$, then $\Gamma \in \tilde{\mathcal{R}}$ if and only if $\Gamma$ has at least two adjacent vertices of degree 2.
\end{corollaryN}


We can also establish that there is a unique graph in $\tilde{\mathcal{R}}$ of $n \geq 5$ vertices (up to isomorphism).

\begin{lemmaN}
For each $n \geq 5$, there is a unique graph in $\tilde{\mathcal{R}}$ with $n$ vertices.
\end{lemmaN}

The preceding lemma is intuitively obvious, and it enables us to establish some helpful notation.

\bigskip
\noindent
\textbf{Definition.}
Let $\tilde{R}_n$ represent the unique graph in $\tilde{\mathcal{R}}$ on $n+5$ vertices.

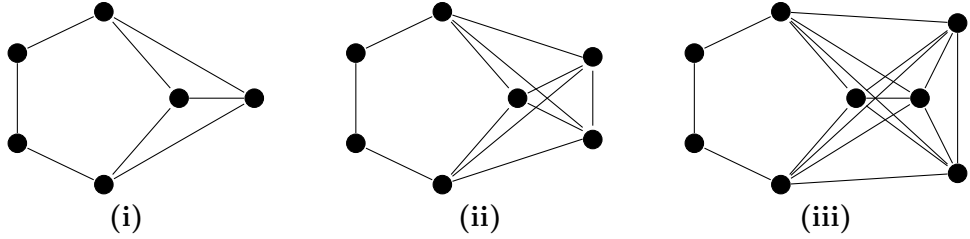
\begin{figure}
   \label{lemma 4.2 figure}
\centering
\scalebox{1.0}{
\begin{tikzpicture}
 
\node (n16) at (11.5,8.1) {};
\node (n17) at (11.5,6.9) {};
\node (n18) at (12.65,8.65) {};
\node (n19) at (12.65,6.35) {};
\node (n20) at (13.65,7.5) {};
\node (n21) at (14.65,7.5) {};
\node (n22) at (16,8.1) {};
\node (n23) at (16,6.9) {};
\node (n24) at (17.15,8.65) {};
\node (n25) at (17.15,6.35) {};
\node (n26) at (18.15,7.5) {};
\node (n27) at (19.15, 8) {};
\node (n28) at (19.15,7) {};
\node (n29) at (20.5,8.1) {};
\node (n30) at (20.5,6.9) {};
\node (n31) at (21.65,8.65) {};
\node (n32) at (21.65,6.35) {};
\node (n33) at (22.65,7.5) {};
\node (n34) at (23.5,7.5) {};
\node (n35) at (24,8.5) {};
\node (n36) at (24,6.5) {};
  
  \draw [fill=black] (11.5,8.1) circle (3.5pt);
\draw [fill=black] (11.5,6.9) circle (3.5pt);
\draw [fill=black] (12.65,8.65) circle (3.5pt);
\draw [fill=black] (12.65,6.35) circle (3.5 pt);
\draw [fill=black] (13.65,7.5) circle (3.5pt);
\draw [fill=black] (14.65,7.5) circle (3.5pt);
 \draw [fill=black] (16,8.1) circle (3.5pt);
\draw [fill=black] (16,6.9) circle (3.5pt);
\draw [fill=black] (17.15,8.65) circle (3.5pt);
\draw [fill=black] (17.15,6.35) circle (3.5 pt);
\draw [fill=black] (18.15,7.5) circle (3.5pt);
\draw [fill=black] (19.15,8.05) circle (3.5pt);
\draw [fill=black] (19.15,6.95) circle (3.5pt);
\draw [fill=black] (20.5,8.1) circle (3.5pt);
\draw [fill=black] (20.5,6.9) circle (3.5pt);
\draw [fill=black] (21.65,8.65) circle (3.5pt);
\draw [fill=black] (21.65,6.35) circle (3.5 pt);
\draw [fill=black] (22.65,7.5) circle (3.5pt);
\draw [fill=black] (23.5,7.5) circle (3.5pt);
\draw [fill=black] (24,8.5) circle (3.5pt);
\draw [fill=black] (24,6.5) circle (3.5pt);

\node at (12.95,5.9) {\textbf{(i)}} {};
\node at (17.65,5.9) {\textbf{(ii)}} {};
\node at (22.25,5.9) {\textbf{(iii)}} {};

\foreach \from/\to in {n16/n17,n16/n18,n17/n19,n18/n20,n19/n20,n18/n21,n19/n21,n20/n21,n22/n23,n22/n24,n23/n25,n24/n26,n25/n26,n27/n28,n26/n27,n26/n28,n24/n27,n24/n28,n25/n27,n25/n28,n29/n30,n29/n31,n30/n32,n33/n31,n33/n32,n33/n34,n34/n35,n34/n36,n35/n36,n34/n32,n34/n31,n35/n33,n35/n32,n35/n31,n36/n33,n36/n32,n36/n31}
    \draw (\from) -- (\to);
\end{tikzpicture}}
\caption{Above are examples of the graphs $\tilde{R}_1$, $\tilde{R}_2$, and $\tilde{R}_3$ (as (i), (ii), and (iii), respectively). Note how they all contain two vertices of degree 2 and one $K^-_{n+3}$ induced subgraph.}
\end{figure}

\bigskip
Before we proceed to study the eigenvalues of graphs in $\tilde{R}$, we will make the following connection between $\tilde{R}_n$ and suspension graphs.

\begin{lemmaN}
\label{R_n isomorphic S(n-3,2)}
$\tilde{R}_n \cong S(n+2,2)$.
\end{lemmaN}

\begin{proof}
First, we know $\tilde{R}_0$ is $C_5$, which is isomorphic to $S(2,2)$. Next, suppose $n>5$. We will prove the lemma by providing an explicit isomorphism. Consider the function $\phi: V(\tilde{R}_n) \rightarrow V(S(n+2,2))$. Let $a_1, a_2 \in V(\tilde{R}_n)$ be the two vertices of degree 2, which we know exist by Corollary \ref{degree 2 corollary}. Let $a_3,a_4 \in V(\tilde{R}_n)$ be the vertices adjacent to $a_1$ and $a_2$ of degree $>2$, respectively. Let $a_i$, where $5\leq i \leq n$ be the remaining vertices in $V(\tilde{R}_n)$. Thus, we find $\tilde{R}_n[a_3,a_4,a_{5\leq i \leq n+5}]$ is the single maximal $K^-_{n+3}$ induced subgraph, where $a_3a_4$ is the missing edge. Next, label the two vertices of degree 2 in $V(S(n+2,2))$ as $a_1'$ and $a_2'$. These are the two vertices in the $K_2$ subgraph in $B_{n+2,2}$. Next, let $a_3'$ be the vertex of degree $>2$ adjacent to $a_1'$ and similarly define $a_4'$, with respect to $a_2'$. Finally, label the remaining vertices $a_i'$, where $5\leq i \leq n+5$. Note that these are the vertices in the $K_{n+2}$ subgraph of $B_{n+2,2}$, excluding the bridge vertex in $K_{n-3}$. We assert $\phi(a_j) = a_j'$ is a graph isomorphism (regardless of how we label the vertices $a_i$ and $a'_i$ where $5 \leq i \leq n+5$), which establishes the lemma.
\end{proof}

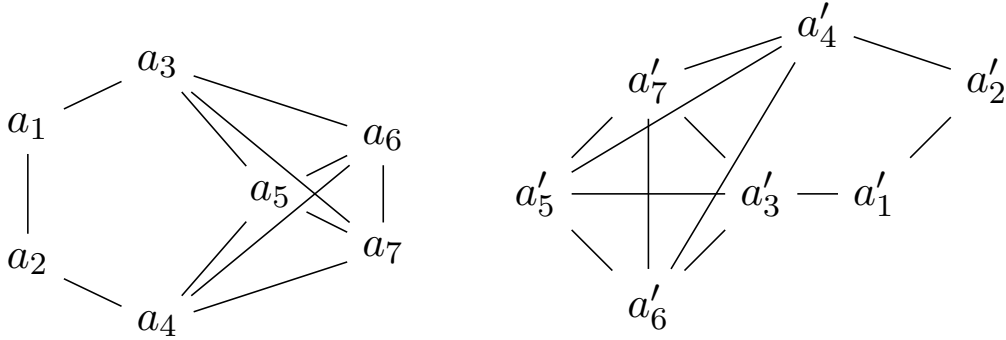
\begin{figure}
    \centering
    \scalebox{1.5}{
 \begin{tikzpicture}   
\node (n8) at (21.5,8.5) {$a'_7$};
  \node (n9) at (22.5,7.5)  {$a'_3$};
  \node (n10) at (23.5,7.5)  {$a'_1$};
  \node (n11) at (24.5,8.5) {$a'_2$};
  \node (n13) at (21.5,6.5)  {$a'_6$};
  \node (n14) at (20.5,7.5) {$a'_5$};
  \node (n15) at (23,9) {$a'_4$};
  \node (n22) at (16,8.1) {$a_1$};
\node (n23) at (16,6.9) {$a_2$};
\node (n24) at (17.15,8.65) {$a_3$};
\node (n25) at (17.15,6.35) {$a_4$};
\node (n26) at (18.15,7.5) {$a_5$};
\node (n27) at (19.15, 8) {$a_6$};
\node (n28) at (19.15,7) {$a_7$};


\foreach \from/\to in {n22/n23,n22/n24,n23/n25,n24/n26,n25/n26,n27/n28,n26/n27,n26/n28,n24/n27,n24/n28,n25/n27,n25/n28,n14/n8,n14/n13,n14/n9,n9/n8,n9/n13,n8/n13,n9/n10,n10/n11,n15/n8,n15/n11,n15/n13,n15/n14}
    \draw (\from) -- (\to);
   
   \end{tikzpicture}}
    \caption{An example of one of the isomorphisms between $\tilde{R}_2$ and $S(4,2)$.}
    \label{fig:my_label}
\end{figure}

Next, we will find upper and lower bounds of the greatest eigenvalue $\lambda_1$ of graphs in $\tilde{R}$, just as we did for $B_{m,m-1}$ graphs.

\begin{lemmaN}
\label{lambda 1 eigenvalue bounds for tilde R}
For the graph $\tilde{R}_n$, $\frac{(n+1)(n+4)}{n+3} \leq \lambda_1 \leq n+2$
\end{lemmaN}

\begin{proof}
This follows from the fact that the greatest eigenvalue of a graph is bounded above by the greatest vertex degree and below by the average degree of any induced subgraph. It is not difficult to verify that the greatest vertex degree of $\tilde{R}_n$ is $n+2$ and the greatest average degree of any induced subgraph is on the $K^-_{i \geq 4}$ induced subgraph, where the average vertex degree is \[\frac{(n+1)(n+2)+2(n+1)}{n+3} = \frac{(n+1)(n+4)}{n+3}.\]
This establishes what we sought to prove.
\end{proof}

Let $\varphi \approx 1.618$ denote the golden ratio. In the following the theorem, we will draw a connection between this number and graphs in $\tilde{R}$.

\begin{theoremN}
\label{golden ratio eigenvalues for reseminant graphs}
$-\varphi$ and $\varphi^{-1}$ are eigenvalues of $\tilde{R}_n$ for all $n \in \mathbb{N}$.
\end{theoremN}

\begin{proof}
We will prove this by induction on the number of vertices, and we will consider the two eigenvalues separately. First, we will prove that $-\varphi$ is an eigenvalue of the adjacency matrices of these graphs in $\tilde{R}$.
\\
\\
First, we will prove inductively that if $v_i$ is the vertex on the induced 5-cycle that is repeatedly duplicated, then the rows of $A(\tilde{R}_n)$ corresponding to the other vertices in the induced 5-cycle are linearly dependent. Label the five vertices of the induced 5-cycle 1,2,3,4, and 0. They were written in this order because the ordering of these vertices corresponds to which row and column they represent in the adjacency matrix (ie. vertex 1 is the first row/column, vertex 2 is the second row/column, ..., but vertex 0 is the fifth row/column). If we let $[-\varphi I -A(\tilde{R}_n)]_j$ represent the $j$th row of $-\varphi I -A(\tilde{R}_n)$, then we assert
\begin{equation}
\label{linear dependence golden ratio}
[-\varphi I -A(\tilde{R}_n)]_{i-1} = [-\varphi I -A(\tilde{R}_n)]_{i+1} + \varphi[-\varphi I -A(\tilde{R}_n)]_{i-2}-\varphi[-\varphi I -A(\tilde{R}_n)]_{i+2}
\end{equation}
with $i$ in mod 5, where we are repeatedly duplicating the $i$th vertex of the induced 5-cycle. Thus, as an example, we find $[-\varphi I -A(\tilde{R}_n)]_0 = [-\varphi I -A(\tilde{R}_n)]_2 + \varphi[-\varphi I -A(\tilde{R}_n)]_4-\varphi[-\varphi I -A(\tilde{R}_n)]_3$ if we are duplicating vertex 1. As our base case, we show that this holds for $\tilde{R}_1$.
\\
\\
It is not difficult to verify that all the possible variations of Equation \ref{linear dependence golden ratio} are satisfied on the submatrix of $-\varphi I -A(\tilde{R}_n)$ for the vertices of the induced 5-cycle, which tells us that $-\varphi$ is an eigenvalue of $\tilde{R}_0$, or $C_5$. This submatrix is shown below:
\[
\left(\begin{matrix}
-\varphi & -1 & 0 & 0 & -1\\
-1 & -\varphi & -1 & 0 & 0\\
0 & -1 & -\varphi & -1 & 0\\
0 & 0 &-1 & -\varphi & -1\\
-1 & 0 & 0 & -1 & -\varphi
\end{matrix}\right).
\]
Thus, to show that the precise variation of Equation \ref{linear dependence golden ratio} holds for $\tilde{R}_1$, we just need to show that it holds for the sixth coordinate in the rows. If we are duplicating the $i$th vertex, then the $(i-1)$th and $(i+1)$th vertices (mod 5) are both -1 in $-\varphi I - A(\tilde{R}_1)$, while the other two vertices, the (i-2)th and (i+2)th (mod 5), are both 0. From this, we deduce that Equation \ref{linear dependence golden ratio} determines the linear dependence of some of the rows of $-\varphi I - A(\tilde{R}_1)$, so $\det{-\varphi I - A(\tilde{R}_1)} = 0$, which tells us that $-\varphi$ is an eigenvalue of $\tilde{R}_1.$ This establishes the base case.
\\
\\
Next, suppose Equation \ref{linear dependence golden ratio} holds for all $\tilde{R}_j$ for $j \leq n$, where the $i$th vertex on the induced 5-cycle is repeatedly duplicated. This tells us $\det{-\varphi I - A(\tilde{R}_j)}=0$, so $-\varphi$ is an eigenvalue of $A(\tilde{R}_j)$. Now, consider $\tilde{R}_{n+1}$. Because we are duplicating the same vertex that was duplicated for generating all of the $\tilde{R}_j$ before $\tilde{R}_{n+1}$, we know that the $(n+1)$th coordinate positions of the first five rows will be identical to the $l$th coordinate positions of these same rows for $6\leq l \leq n$. From this, we can deduce that Equation \ref{linear dependence golden ratio} will be unaffected by repeatedly duplicating the same vertex, so we find that it still holds for $\tilde{R}_{n+1}$. We can conclude that $-\varphi$ is an eigenvalue of $\tilde{R}_{n+1}$. It follows by induction that $-\varphi$ is an eigenvalue of every graph in $\tilde{\mathcal{R}}$.
\\
\\
We can use similar reasoning to show that $\varphi^{-1}$ is an eigenvalue of any graph in $\tilde{\mathcal{R}}$. The equation for establishing the linear dependence of the rows in $\varphi^{-1}I - A(\tilde{R}_n)$ is 
\begin{equation}
    [\varphi^{-1}I - A(\Gamma_n)]_{i-1} = [\varphi^{-1}I-A(\Gamma_n)]_{i+1} + \varphi^{-1}[\varphi^{-1}I-A(\Gamma_n)]_{i+2} - \varphi^{-1}[\varphi^{-1}I-A(\Gamma_n)]_{i-2}
\end{equation}
where, once again, $i$ is in mod 5, and the $i$th vertex is being duplicated. From this, we can deduce that $\varphi^{-1}$ is an eigenvalue of any graph in $\tilde{\mathcal{R}}$. Therefore, we have proved the theorem.
\end{proof}

From Lemma \ref{-1 eigenvalue reseminant graphs} and Theorem \ref{golden ratio eigenvalues for reseminant graphs}, we are left with all but three of the eigenvalues for any $\tilde{R}_n$. We will let $\theta_1, \theta_2,$ and $\theta_3$ be these remaining eigenvalues, such that $\theta_1 \geq \theta_2 \geq \theta_3$. By Lemma \ref{lambda 1 eigenvalue bounds for tilde R}, we see $\lambda_1 = \theta_1$. Just as we did in Lemma \ref{bounds of lambda_3 and lambda_4 B_m,m-1}, we can find bounds for the sum and product of $\theta_2$ and $\theta_3$, but first we will determine the characteristic polynomial of $\tilde{R}_n$.

\begin{theoremN}
\label{char polynomial for tilde R}
The characteristic polynomial of $\tilde{R}_n$ is
\begin{equation}
\label{characteristic polynomial R_n}
    \phi(A(\tilde{R}_n), x) = (x^3-(n+1)x^2-(n+3)x+(3n+2))(x+1)^n(x^2+x-1).
\end{equation}
\end{theoremN}

\begin{proof}
By Lemmas \ref{-1 eigenvalue reseminant graphs} and \ref{golden ratio eigenvalues for reseminant graphs}, we know that $(x+1)^{n-5}$ and $x^2+x-1$ divide $\phi(A(\tilde{R}_n), x)$. Thus, we know that there are only the three remaining eigenvalues $\theta_1, \theta_2,$ and $\theta_3$. Just as we did in the proof of Theorem \ref{characteristic polynomial equation for B_m,m-1}, we only need to calculate $\sum_{i=1}^3 \theta_i$, $\sum_{1 \leq i < j \leq 3} \theta_i \theta_j$, and $\prod_{i=1}^3 \theta_i$, and then we can construct the polynomial $\prod_{i=1}^3 (x-\theta_i)$.
\\
\\
First, we will show $\sum_{i=1}^3 \theta_i = n+1$. We know $\sum_{i=1}^n \lambda_i = 0$, for this is just the trace of $A(\tilde{R}_n)$, and
\[
\sum_{i=1}^n \lambda_i = \sum_{i=1}^3 \theta_i -n-\varphi + \varphi^{-1} = \sum_{i=1}^3\theta_i - (n+1) = 0,
\]
which establishes $\sum_{i=1}^3 \theta_i = n+1$. Next, we need to show $\sum_{1\leq i < j \leq 3}\theta_i \theta_j = -(n+3)$. Recall from the proof of Theorem \ref{characteristic polynomial equation for B_m,m-1} that $\sum_{1\leq i < j \leq n}\lambda_i \lambda_j = -e(\Gamma)$, where $e(\Gamma)$ is the number of edges in the graph $\Gamma$ being studied. Hence, for our case at hand, this sum is $-e(\tilde{R}_n)$, which is not difficult to determine. Any graph in $\tilde{\mathcal{R}}$ has three edges distinct from the edges in $K^-_{n+3}$, so $e(\tilde{R}_n) = 3+\frac{1}{2}(n+2)(n+3)-1 = 2 + \frac{1}{2}(n+2)(n+3)$. Thus, we have found $\sum_{1\leq i < j \leq n}\lambda_i \lambda_j = -(2 + \frac{1}{2}(n+2)(n+3))$. Next, we find 
\[
\sum_{1\leq i < j \leq n}\lambda_i \lambda_j = \sum_{1\leq i < j \leq 3}\theta_i \theta_j  + (-\varphi+\varphi^{-1}-n)\sum_{i=1}^3 \lambda_3\]\[+n(\varphi-\varphi^{-1})-1+ \frac{1}{2}(n)(n-1) = \sum_{1\leq i < j \leq 3}\theta_i \theta_j  -(n+1)^2\]\[+(n-1)+ \frac{1}{2}(n)(n-1) =  -(2 + \frac{1}{2}(n+2)(n+3)).
\]
When we solve this for $\sum_{1\leq i < j \leq 3}\lambda_i \lambda_j$, we find $\sum_{1\leq i < j \leq 3}\lambda_i \lambda_j = -(n+3)$. Now, we need to determine $\prod_{i=1}^3 \theta_i$. By Lemma \ref{R_n isomorphic S(n-3,2)}, we know that $\tilde{R}_n$ is isomorphic to $S(n+2,2)$, so we can easily determine the determinate of its adjacency matrix by Theorem \ref{determinant C -> G}. Thus, we know $\det{A(\tilde{R}_n)} = \prod_{i=1}^n \lambda_i = (-1)^{n+4}(8(n+2)-5(n+4)+6) = (-1)^n(3n+2)$. Additionally,
\[
\prod_{i=1}^n \lambda_i = \prod_{i=1}^3 \theta_i(-1)^n(-\varphi)(\varphi^{-1}) = \prod_{i=1}^3 \theta_i(-1)^{n+1} = (-1)^n(3n+2),
\]
so we find $\prod_{i=1}^3 \theta_i = -(3n+2)$. Putting all of these results back together, we find $\prod_{i=1}^3 (x-\theta_i) = (x^3-(n+1)x^2-(n+3)x+(3n+2))$, so Equation \ref{characteristic polynomial R_n}, which is essentially $\phi(A(\tilde{R}_n),x) = \prod_{i=1}^n (x-\lambda_i)$, is the $n$th degree polynomial with the $n$ eigenvalues of $A(\tilde{R}_n)$ as its roots. Hence, it is the characteristic polynomial and our proof is complete.

\end{proof}

Now, we are ready to find upper and lower bounds for the sum and product of $\theta_2$ and $\theta_3$. 

\begin{lemmaN}\label{theta 2 and 3 inequalities tilde R}
The eigenvalues $\theta_2$ and $\theta_3$ of $\tilde{R}_n$ satisfy the following inequalities:
\begin{enumerate}
    \item $-1 \leq \theta_2 + \theta_3 \leq -\frac{n+1}{n+3}$
    \item $-\frac{(2+3n)(n+3)}{(n+1)(n+4)} \leq \theta_2 \theta_3 \leq -\frac{2+3n}{n+2}$.
\end{enumerate}
\end{lemmaN}

\begin{proof}
This follows from Lemma \ref{lambda 1 eigenvalue bounds for tilde R} and the equations for $\sum_{i=1}^3 \theta_i$ and $\prod_{i=1}^3 \theta_i$, which were found in the proof of Theorem \ref{char polynomial for tilde R}.
\end{proof}
In the following theorem, we establish the graph spectrum of $\tilde{R}_n$.

\begin{theoremN}
If $n>0$, $\theta_1 > \theta_2> \varphi^{-1} > -1> -\varphi > \theta_3 $, and \[Spec(\tilde{R}_n) = \left(\begin{matrix}
\theta_1 & \theta_2 & \varphi^{-1} & -1 & -\varphi & \theta_3\\
1&1&1& n-5 &1&1
\end{matrix}\right).\] 
If $n=0$, $\theta_3 = -\varphi$, $\theta_2 = \varphi^{-1}$, and $\theta_1 = 2$, so 
\[
Spec(\tilde{R}_0) = Spec(C_5) = \left(\begin{matrix}
2 & \varphi^{-1}  & -\varphi\\
1&2& 2
\end{matrix}\right).
\]
\end{theoremN}

\begin{proof}
First, suppose $n=0$. In this case, the $\theta_i$ are roots of the polynomial $x^3-x^2-3x+2$, which factors into $(x-2)(x^2+x-1)$, so we find $\theta_1 =2, \theta_2 = \varphi^{-1}$, and $\theta_3 = -\varphi$. Along with Theorem \ref{char polynomial for tilde R}, this establishes the second part of the theorem.
\\
\\
Now, we will let $n>0$. By Theorem \ref{char polynomial for tilde R}, we note a priori that $\varphi^{-1}$, and $-\varphi$ are eigenvalues in the spectrum with multiplicity $\geq 1$ and $-1$ has multiplicity $\geq n-5$. By Lemma \ref{lambda 1 eigenvalue bounds for tilde R}, we see that $\theta_1$ increases along with $n$, which will affect the possible values of $\theta_2$ and $\theta_3$. Also, we deduce in general that $\theta_1 > \varphi^{-1}$, and we know that $\theta_1$ must be the unique largest eigenvalue by Lemma \ref{PF-lemma}. We know $\prod_{i=1}^3 \theta_i = -(3n+2) < 0$, so $\theta_2$ and $\theta_3$ cannot both be negative. Thus, $\theta_2>0$ and $\theta_3<0$ for all $n$. Next, assume there is some $n$ such that $\theta_2 < \varphi^{-1}$. By Lemma \ref{theta 2 and 3 inequalities tilde R} (i), we then see that $\theta_3> -\varphi$. However, then we find $-1<\theta_2 \theta_3<0$. However, this contradicts the second part of Lemma \ref{theta 2 and 3 inequalities tilde R} because $-\frac{2+3n}{n+2}<-1$ for all $n>0$, so we know that there is no $n>0$ such that $\theta_2 < \varphi^{-1}$. Hence, $\theta_1 > \theta_2 > \varphi^{-1}$ for all $n>0$. Now, we just need to show that there is no $n>0$ such that $0>\theta_3>-\varphi$. Assume the contrary. Once again, by Lemma \ref{theta 2 and 3 inequalities tilde R}, this requires that $\theta_2<\varphi^{-1}$, which we just showed is not the case. Thus, we have another contradiction. Assume $\theta_2 = \varphi^{-1}$ for some $n>0$. The characteristic polynomial must have all integer coefficients, and the minimal polynomial of $\varphi^{-1}$ in the polynomial ring $\mathbb{Z}[x]$ is $x^2+x-1$, so we deduce $x^2+x-1$ must divide $x^3+(4-n)x^2+(2-n)x+(3n-13)$. Hence, we also know $\theta_3 = -\varphi$, and there must be some $k \in \mathbb{Z}$ such that $(x^2+x-1)(x-k) = x^3-(n+1)x^2-(n+3)x+(3n+2)$. This simplifies to the system of equations $-k+n=-2$ and $k-3n=2$, which has the unique solution $(n,k)=(0,2)$. This contradicts our assumption $n>0$. We get the same result if we assume first that $\theta_3 = -\varphi$, so we know that the $\theta_i$ are all distinct from each other and $\varphi^{-1}, -\varphi,$ and $-1$, so we conclude that the theorem is proved.
\end{proof}

The natural question to ask after finding the spectrum of a graph is: \textit{is this graph determined by its spectrum?} In other words, does the spectrum provide a characterization, or are there other non-isomorphic graphs that share the same spectrum. Many articles have considered this question for different families of graphs (See \cite{almost_complete_spectra}, \cite{DS_graphs}, \cite{pineapple_spectrum}, \cite{dumbell_spectra}, \cite{isolated_vertex_spectrum}). Given the high multiplicities of the -1 eigenvalues for both $B_{m,m-1}$ and $\tilde{R}_n$, it may be helpful to approach this problem by considering structural equivalence in graphs and determining which graph structures can be candidates for cospectral graphs with $B_{m,m-1}$ and $\tilde{R}_n$. Structural equivalence is closely related with the multiplicity of the -1 eigenvalue; this problem was studied in \cite{seq_paper}.

\section{Acknowledgements}

This research was conducted at Texas State University under NSF-REU grant DMS-1757233 during the summer of 2020. The first four authors thank NSF and the fifth author gratefully acknowledges the financial support from the Office of Undergraduate Research at Washington University in St. Louis. The authors thank Texas State University for running the REU online during this difficult period of social distancing and providing a welcoming and supportive work environment. In particular, Dr. Yong Yang, the director of the REU program, is recognized for conducting an inspired and successful research program. The first, second, third and fifth authors also thank their mentor, the fourth author, for his invaluable advice and guidance throughout this project. The second author is largely responsible for the results of this paper, with guidance from the fourth author. The other authors worked on other projects during the summer REU program. 
\printbibliography
\end{document}